\documentclass[12pt]{amsart}
\usepackage{import}
\usepackage{mypreamble}
\usepackage{subfiles}

\title{Weighted Orlicz-Poincar\'e inequalities in product spaces}
\author{Lucas Yong}
\address{Department of Mathematics \\
          National Louis University \\
          Chicago, IL 60603}
\email{\href{mailto:lyong@nl.edu}{lyong@nl.edu}}

\begin{document}
\begin{abstract}
This article is a follow-up to \cite{korobenko_milshstein_yong}. We establish necessary and sufficient conditions for weighted Orlicz-Poincar\'e inequalities in product spaces. These results follow the work of Chua and Wheeden \cite{chuawheeden}, who established similar results for weighted Poincar\'e inequalities in product spaces.
\end{abstract}
\maketitle
\section{Introduction}
Interest in Sobolev and Poincar\'e inequalities arises, in particular, from their applications in regularity theory for partial differential equations. In \cite{FKS}, weighted Sobolev and Poincar\'e inequalities are used to perform the Moser iteration scheme to show H\"older continuity of weak solutions to certain classes of degenerate elliptic equations. There have also been promising results in Orlicz-Sobolev settings, where the underlying Lebesgue $L^q$ spaces are replaced by the more general Orlicz $L^\Phi$ spaces, where $\Phi$ is a Young function. An example of such a result is found in \cite{KRSSh}, where it is shown that a version of an Orlicz-Sobolev inequality can be used in a modification of the DeGiorgi iteration scheme to show continuity of weak solutions to infinitely-degenerate elliptic equations. The usefulness of these kinds of inequalities makes it natural to seek sufficient conditions for them to hold for a set of test functions (e.g. Lipschitz continuous functions). 

However, there are few such conditions that are easy to verify. In \cite{chuawheeden}, neccessary and sufficient conditions are given for weighted $1$-dimensional Poincar\'e inequalities to hold for Lipschitz continuous functions, and these are also adapted to product spaces. More recently, it was shown in \cite{korobenko_milshstein_yong} that a more general version of this result holds in the $1$-dimensional setting. That is, the authors instead considered an \emph{Orlicz}-Poincar\'e inequality, replacing the Lebesgue $L^q$ norm on the lefthand side by the more general $L^\Phi$ gauge norm (defined with respect to a Young function $\Phi$ satisfying certain properties). The main results of \cite{korobenko_milshstein_yong} are necessary and sufficient conditions on weights that are easy to verify, analogous to those in \cite{chuawheeden}. 

The main goal of the present article is to extend the results of \cite{korobenko_milshstein_yong} to obtain necessary and sufficient conditions for weighted Orlicz-Poincar\'e inequalities in \emph{product spaces}, as was done in \cite{chuawheeden} in the classical case. The significance of such conditions lies in the fact that $n$-dimensional Orlicz-Sobolev inequalities can potentially be used to develop regularity theory for degenerate partial differential equations.

We now describe our main results. While what follows is valid in $n$ dimensions, we restrict our attention to the case $n=2$ to simplify notation. 

Throughout, let $I = [a, b], J = [c, d] \subset \R$, and let $\mu = \mu_1 \times \mu_2, \nu = \nu_1\times \nu_2, w = w_1 \times w_2$ be product weights on $I \times J$ such that $\mu, \nu \in L^1(I \times J)$. Let $1 \leq p_1, p_2, s_1 < \infty$, and let $\Phi$ be a Young function.

Before stating our main results, we introduce some adaptions of notation from \cite{korobenko_milshstein_yong}. Let
\[
K_{1, \Phi}(\mu_1, \nu_1, w_1) 
= \widetilde{K}_{1, \Phi}(\mu_1, \nu_1, w_1)
= \frac{1}{\nu_1(I)}
    \sup_{a<x<b}\left\{
    \frac{1}{w_1(x)}
    \left\|
    \nu_1[a,x]\chi_{[x,b]}
    -
    \nu_1[x,b]\chi_{[a,x]}
    \right\|_{L_{\mu_1}^\Phi(I)}
    \right\}.
\]
For $p_1 > 1$, let
\[
    \begin{aligned}
    K_{p_1,\Phi}(\mu_1, \nu_1, w_1) = 
    \frac{1}{\nu_1(I)}
    &\Biggl(
    \sup_{a<x<b}\biggl\{
    \left[\Phi^{-1}\left(\frac{1}{\mu_1[x,b]^{1/2}}\right)\right]^{-2}\biggl(\int_a^x \nu_1[a,t]^{{p_1}'}w_1(t)^{1-{p_1}'} \dd t\biggr)^{1/{p_1}'}
    \biggr\}\\
    &+
    \sup_{a<x<b}\biggl\{
    \left[\Phi^{-1}\left(\frac{1}{\mu_1[a,x]^{1/2}}\right)\right]^{-2}\biggl(
    \int_x^b \nu_1[t,b]^{{p_1}'} w_1(t)^{1-{p_1}'} \dd t
    \biggr)^{1/{p_1}'}
    \biggr\}
    \Biggr),
    \end{aligned}
\]
and
\[
\begin{aligned}
\tilde{K}_{p_1,\Phi}(\mu_1, \nu_1, w_1) := 
\frac{1}{\nu_1(I)}
&\Biggl(
\sup_{a<x<b}\biggl\{
\left[\Phi^{-1}\left(\frac{1}{\mu_1[x,b]}\right)\right]^{-1}\biggl(\int_a^x \nu_1[a,t]^{{p_1}'}w_1(t)^{1-{p_1}'} \dd t\biggr)^{1/{p_1}'}\biggr\}\\
&+
\sup_{a<x<b}\biggl\{
\left[\Phi^{-1}\left(\frac{1}{\mu_1[a,x]}\right)\right]^{-1}\biggl(
\int_x^b \nu_1[t,b]^{{p_1}'} w_1(t)^{1-{p_1}'} \dd t
\biggr)^{1/{p_1}'}
\biggr\}
\Biggr).
\end{aligned}
\]
Here, ${p_1}'$ is the H\"older conjugate of ${p_1}$. 

\begin{remark}
The constants $K_{p_1,\Phi}(\mu_1, \nu_1, w_1)$ and $\tilde{K}_{p_1,\Phi}(\mu_1, \nu_1, w_1)$ are similar; for example, they are equal when $\Phi(t) = |t|^q$. Note that it is always the case that
\[
K_{p_1,\Phi}(\mu_1, \nu_1, w_1) \geq \tilde{K}_{p_1,\Phi}(\mu_1, \nu_1, w_1).
\]
Closing the gap between these constants would strengthen the results of \cite{korobenko_milshstein_yong} and the present article, but we suspect that different methods from those employed in \cite{korobenko_milshstein_yong} must be used to achieve this.

Of course, we may similarly define constants $K_{p_2,\Phi}(\mu_2, \nu_2, w_2)$ and $\tilde{K}_{p_2,\Phi}(\mu_2, \nu_2, w_2)$, which depend on the interval $J$.
\end{remark}

Our main results are as follows.
\newpage
\begin{theorem}
\label{thm:two_dim_sufficiency}
Suppose that $\Phi$ is submultiplicative and invertible on $[0, \infty)$, and that $\Gamma_i(t) :=  \Phi\left(t^{1/{p_i}}\right)$ is convex for $i = 1, 2$. If $K_{p_i,\Phi}(\mu_i, \nu_i, w_i) < \infty$ for $i = 1,2$, then
\begin{align*}
    &\left\|f - \frac{1}{\nu(I \times J)}\int_{I \times J} f(y_1, y_2) \;\dd\nu_1(y_1)\dd\nu_2(y_2)\right\|_{L_\mu^\Phi(I \times J)}\\
    \leq&\;\;
    C_1\cdot\left\|\frac{\partial f}{\partial x_1}\right\|_{L_{w_1 \times \mu_2}^{(p_1,\Phi)}(I \times J)}
    C_2 \cdot \frac{2 \left[\Phi^{-1}\left(\frac{1}{\mu_1(I)}\right)\right]^{-1}}{\nu_1(I)^{s_1}}\cdot\left\|\frac{\partial f}{\partial x_2}\right\|_{L_{\nu_1 \times w_2}^{(\widehat{s_1, p_2})}(I \times J)}
\end{align*}
for all Lipschitz continuous functions $f \colon I \times J \to \R$. Moreover, if $C_i$ are the best possible constants for $i = 1,2$, then
\[
\begin{aligned}
C_i &= K_{1, \Phi}(\mu_i, \nu_i, w_i), \quad&\text{if $p_i = 1$};\\
C_i &\leq C_0(\Phi)\cdot K_{p_i, \Phi}(\mu_i, \nu_i, w_i), \quad&\text{if $p_i > 1$}.\\
\end{aligned}
\]
Here, $C_0(\Phi) := 2 \left[\Phi^{-1}\left(\frac{1}{2}\right)\right]^{-1}$.
\end{theorem}
\begin{theorem}
\label{thm:two_dim_necessity}
Suppose that $\Phi$ is invertible on $[0, \infty)$. For $i = 1, 2$, if there exists constants $C_i > 0$ such that
\begin{align}
\label{eqn:sufficiencythm}
\begin{split}
    &\left\|f - \frac{1}{\nu(I \times J)}\int_{I \times J} f(y_1, y_2) \;\dd\nu_1(y_1)\dd\nu_2(y_2)\right\|_{L_\mu^\Phi(I \times J)}\\
    \leq&\;\;
    C_1\cdot\left\|\frac{\partial f}{\partial x_1}\right\|_{L_{w_1 \times \mu_2}^{(p_1,\Phi)}(I \times J)}
    \;+\;
    C_2 
    \cdot \frac{2 \left[\Phi^{-1}\left(\frac{1}{\mu_1(I)}\right)\right]^{-1}}{\nu_1(I)^{s_1}}
    \cdot\left\|\frac{\partial f}{\partial x_2}\right\|_{L_{\nu_1 \times w_2}^{(\widehat{s_1, p_2})}(I \times J)}
\end{split}
\end{align}
for all Lipschitz continuous functions $f \colon I \times J \to \R$, then $\tilde{K}_{p_i,\Phi}(\mu_i, \nu_i, w_i) < \infty$.
\end{theorem}

The norms utilized in the inequalities in both the above theorems will be defined in the next section.

The proofs of these results closely follow the arguments in \cite[Section 3]{chuawheeden}, however, the adaptation from Lebesgue spaces to Orlicz spaces is not straightforward, and there are obstacles to overcome. For example, the gauge norm on the Orlicz space $L_\mu^\Phi(I \times J)$ is not generally identical to the iteration of the one-dimensional gauge norms on each product space, unlike the Lebesgue norm on $L_\mu^q(I \times J)$.

\textbf{Acknowledgements.} The author is extremely grateful to Luda Korobenko for her unwavering support and helpful comments throughout the duration of this project.
\newpage
\section{Preliminaries}
\label{sec:prelim}
We first review some important aspects of the theory of Orlicz spaces. We refer the reader to \cite{raoren} for a comprehensive introduction.
\begin{defn}
A \textbf{Young function} is a convex function $\Phi\colon \R \to [0, \infty)$ such that
\begin{enumerate}[label=(\roman*)]
    \item $\Phi$ is even, i.e. $\Phi(-t) = \Phi(t)$,
    \item $\Phi(0) = 0$,
    \item $\lim_{t \to \infty} \Phi(t) = +\infty$.
\end{enumerate}
\end{defn}
\begin{defn}
Let $f \colon I \times J \to \R$ be a measurable function. Define
\[
\|f\|_{L_\mu^\Phi(I \times J)} := \inf\left\{k > 0 : \int_\Omega \Phi\left(\frac{f}{k}\right)\dd \mu \leq 1\right\}.
\]
$\|\cdot\|_{L_\mu^\Phi(I \times J)}$ is called the \textbf{gauge norm} (or \emph{Luxemburg} norm).
\end{defn}
\begin{defn}
Let
\[
L_\mu^\Phi(I \times J) := \left\{f \colon I \times J \to \R: f \text{ is measurable and } \|f\|_{L_\mu^\Phi} < \infty\right\}.
\]
\end{defn}
The reader may verify that the gauge norm $\|\cdot\|_{L_\mu^\Phi(I \times J)}$ is indeed a norm which makes $L_\mu^\Phi(I \times J)$ a Banach space.

The following Lemma from \cite{korobenko_milshstein_yong} is a generalization of Minkowski's inequality for integrals, which will be important in the proofs contained in the next section.
\begin{lemma}
\label{lemma:gen_minkowski}
Let $F \colon \R \times \R \to \R$ be a measurable function. Then,
\[
\left\|\int F(\bullet,t) \dd t\right\|_{L_\mu^\Phi(\R)} 
\leq 2\int \left\|F(\bullet,t)\right\|_{L_\mu^\Phi(\R)}  \dd t
\]
\end{lemma}
The reader may refer to \cite[Lemma 2.8]{korobenko_milshstein_yong} for a proof of this result.

Finally, we conclude this section by introducing adaptions of some notation from \cite{chuawheeden}. 
\begin{defn}
\label{def:repeated_norm}
Let $F \colon I \times J \to \R$ be a measurable function. Let
\[
\left\|F\right\|_{L_{\mu_1 \times \mu_2}^{(p_1,p_2)}(I \times J)}
:=
\left\|\left\|F\right\|_{L_{\mu_1}^{p_1}(I)}\right\|_{L_{\mu_2}^{p_2}(J)}.
\]
\end{defn}
\begin{remark}
\label{remark:repeatednorms}
The use of repeated norms in the definition above requires some care. For $t \in J$, define $F_{t} \colon I \to \R$ by $F_{t}(x_1) = F(x_1, t)$. Now define
\[
\begin{aligned}
    \left\|F_{\bullet}\right\|_{L_{\mu_1}^{p_1}(I)} \colon J &\longrightarrow [0, \infty)\\
    x_2 &\longmapsto \left\|F_{x_2}\right\|_{L_{\mu_1}^{p_1}(I)}
\end{aligned}
\]
Now with the $1$-dimensional gauge norm with respect to the second dimension, we may write
\[
\left\|\left\|F\right\|_{L_{\mu_1}^{p_1}(I)}\right\|_{L_{\mu_2}^{p_2}(J)}
:=
\left\|\left\|F_{\bullet}\right\|_{L_{\mu_1}^{p_1}(I)}\right\|_{L_{\mu_2}^{p_2}(J)},
\]
so that \autoref{def:repeated_norm} makes sense.
\end{remark}
\begin{defn}
\label{def:repeated_norm2}
Similar to \autoref{def:repeated_norm}, we define
\[
\left\|F\right\|_{L_{\mu_1 \times \mu_2}^{(\widehat{p_1,p_2})}(I \times J)}
:=
\left\|\left\|F\right\|_{L_{\mu_2}^{p_2}(J)}\right\|_{L_{\mu_1}^{p_1}(I)},
\]
and
\[
\left\|F\right\|_{L_{\mu_1 \times \mu_2}^{(p_1,\Phi)}(I \times J)}
:=
\left\|\left\|F\right\|_{L_{\mu_1}^{p_1}(I)}\right\|_{L_{\mu_2}^{\Phi}(J)}.
\]
\end{defn}

\section{Proofs of the main theorems}
We begin with a Lemma relating the $2$-dimensional gauge norm on $I \times J$ and the repeated gauge norm from \autoref{def:repeated_norm}.
\begin{lemma}
\label{lemma:invert_norm_inequality}
Assume that $\Phi$ is submultiplicative. Let $F \in L_\mu^\Phi(I \times J)$. Then,
\[
\left\|F\right\|_{L_\mu^\Phi(I \times J)} 
\leq 
\left\|\left\|F\right\|_{L_{\mu_1}^\Phi(I)}\right\|_{L_{\mu_2}^\Phi(J)}
\]
\end{lemma}
\begin{proof}
Define $k \colon J \to [0, \infty)$ by $k(t) = \left\|F_{t}\right\|_{L_{\mu_1}^\Phi(I)}$. Here, $F_t$ is the function described in \autoref{remark:repeatednorms}. To prove the stated inequality, it suffices to show that
\[
\int_J\int_I \Phi\left(\frac{F(x_1,x_2)}{
\left\|k\right\|_{L_{\mu_2}^\Phi(J)}}\right) \dd\mu_1(x_1)\dd\mu_2(x_2) \leq 1.
\]
We have
\[
\begin{aligned}
&\int_J\int_I \Phi\left(\frac{F(x_1,x_2)}{\left\|k\right\|_{L_{\mu_2}^\Phi(J)}}\right) \dd\mu_1(x_1)\dd\mu_2(x_2) \\
=&
\int_J\int_I \Phi\left(\frac{F(x_1,x_2)\cdot k(x_2)}{\left\|k\right\|_{L_{\mu_2}^\Phi(J)}\cdot k(x_2)}\right) \dd\mu_1(x_1)\dd\mu_2(x_2)\\
\leq&
\int_J\int_I \Phi\left(\frac{F(x_1,x_2)}{k(x_2)}\right)\cdot\Phi\left(\frac{k(x_2)}{\left\|k\right\|_{L_{\mu_2}^\Phi(J)}}\right)\dd\mu_1(x_1)\dd\mu_2(x_2)\\
&\hspace{4in}\text{(submultiplicativity of $\Phi$)}\\
=&
\int_J\left(\int_I \Phi\left(\frac{F(x_1,x_2)}{k(x_2)}\right)\dd\mu_1(x_1)\right)\cdot\Phi\left(\frac{k(x_2)}{ \left\|k\right\|_{L_{\mu_2}^\Phi(J)}}\right)\dd\mu_2(x_2)\\
\leq&
\int_J \Phi\left(\frac{k(x_2)}{\left\|k\right\|_{L_{\mu_2}^\Phi(J)}}\right)\dd\mu_2(x_2)\\
&\hspace{4in}\text{(by definition of $k$ and the gauge norm)}\\
\leq&\;1\\
&\hspace{4in}\text{(by definition of the gauge norm)}
\end{aligned}
\]
This completes the proof.
\end{proof}

\begin{lemma}
\label{lemma:invert_norm_inequality_second_var}
Assume that $\Phi$ is submultiplicative. If $F \in L_\mu^\Phi(I \times J)$ depends only on the second variable, i.e. $F(x_1,x_2) = g(x_2)$ for some $g \in L_{\mu_2}^\Phi(J)$, then
\[
\Phi^{-1}\left(\mu_1(I)\right)\cdot\left\|g\right\|_{L_{\mu_2}^\Phi(J)}
\leq 
\left\|F\right\|_{L_\mu^\Phi(I \times J)}
\leq 
\left[\Phi^{-1}\left(\frac{1}{\mu_1(I)}\right)\right]^{-1}\cdot \left\| g\right\|_{L_{\mu_2}^\Phi(J)}
\]
\end{lemma}
\begin{remark}
In the case that $\Phi(t) = |t|^q$ for $q\geq 1$, the above inequalities are equalities.
\end{remark}
\begin{proof}[Proof of \autoref{lemma:invert_norm_inequality_second_var}]
The second inequality follows from \autoref{lemma:invert_norm_inequality}. Indeed,
\[
\begin{aligned}
\left\|F\right\|_{L_\mu^\Phi(I \times J)}
&\leq
\left\|\left\|F\right\|_{L_{\mu_1}^\Phi(I)}\right\|_{L_{\mu_2}^\Phi(J)}
&&\text{(\autoref{lemma:invert_norm_inequality})}\\
&= 
\left\|1\right\|_{L_{\mu_1}^\Phi(I)}\cdot\left\|g\right\|_{L_{\mu_2}^\Phi(J)}
\end{aligned}
\]
Note that
\[
\begin{aligned}
    \left\|1\right\|_{L_{\mu_1}^\Phi(I)} :&= \inf\left\{k > 0 : \int_I \Phi\left(\frac{1}{k}\right)\dd \mu_1 \leq 1\right\} \\
    &= \inf\left\{k > 0 : \mu_1(I)\cdot \Phi\left(\frac{1}{k}\right)\leq 1\right\} \\
    &= \inf\left\{k > 0 : \Phi\left(\frac{1}{k}\right)\leq \frac{1}{\mu_1(I)}\right\} \\
    &= \inf\left\{k > 0 : \frac{1}{k}\leq \Phi^{-1}\left(\frac{1}{\mu_1(I)}\right)\right\} \\
    &= \inf\left\{k > 0 : k \geq \left[\Phi^{-1}\left(\frac{1}{\mu_1(I)}\right)\right]^{-1}\right\} \\
    &= \left[\Phi^{-1}\left(\frac{1}{\mu_1(I)}\right)\right]^{-1},
\end{aligned}
\]
which yields
\[
\left\|F\right\|_{L_\mu^\Phi(I \times J)} \leq \left[\Phi^{-1}\left(\frac{1}{\mu_1(I)}\right)\right]^{-1}\cdot \left\| g \right\|_{L_{\mu_2}^\Phi(J)},
\]
as desired. Now we turn to the first inequality. We will show that
\[
\int_J \Phi\left(\frac{g(x_2)}{\left[\Phi^{-1}\left(\mu_1(I)\right)\right]^{-1} \cdot \left\|F\right\|_{L_{\mu}^\Phi(I \times J)}}\right) \dd\mu_2(x_2) \leq 1.
\]

Observe,
\[
\begin{aligned}
&\int_J \Phi\left(\frac{g(x_2)}{\left[\Phi^{-1}\left(\mu_1(I)\right)\right]^{-1} \cdot \left\|F\right\|_{L_{\mu}^\Phi(I \times J)}}\right) \dd\mu_2(x_2)\\
=& 
\frac{1}{\mu_1(I)} \int_I \int_J \Phi\left(\frac{F(x_1,x_2)}{\left[\Phi^{-1}\left(\mu_1(I)\right)\right]^{-1} \cdot \left\|F\right\|_{L_{\mu}^\Phi(I \times J)}}\right) \dd\mu_2(x_2)\dd\mu_1(x_1)\\
\leq& 
\frac{1}{\mu_1(I)} \int_I \int_J \Phi\left(\frac{1}{\left[\Phi^{-1}\left(\mu_1(I)\right)\right]^{-1}}\right)\cdot \Phi\left(\frac{F(x_1,x_2)}{\left\|F\right\|_{L_{\mu}^\Phi(I \times J)}}\right) \dd\mu_2(x_2)\dd\mu_1(x_1)\\
&\hspace{4in}\text{(submultiplicativity of $\Phi$)}\\
=& \frac{1}{\mu_1(I)} \int_I \int_J \mu_1(I)\cdot \Phi\left(\frac{F(x_1,x_2)}{\left\|F\right\|_{L_{\mu}^\Phi(I \times J)}}\right) \dd\mu_2(x_2)\dd\mu_1(x_1)\\
=& \int_I \int_J \Phi\left(\frac{F(x_1,x_2)}{\left\|F\right\|_{L_{\mu}^\Phi(I \times J)}}\right) \dd\mu_2(x_2)\dd\mu_1(x_1)\\
\leq&\; 1.\\
&\hspace{4in}\text{(by definition of the gauge norm)}
\end{aligned}
\]
We have shown that
\[
\left\|g\right\|_{L_{\mu_2}^\Phi(J)} \leq  \left[\Phi^{-1}\left(\mu_1(I)\right)\right]^{-1}\cdot\left\|F\right\|_{L_\mu^\Phi(I \times J)},
\]
which implies
\[
\Phi^{-1}\left(\mu_1(I)\right)\cdot\left\|g\right\|_{L_{\mu_2}^\Phi(J)} \leq  \left\|F\right\|_{L_\mu^\Phi(I \times J)},
\]
proving the first inequality. This completes the proof.
\end{proof}

\begin{lemma}
\label{lemma:invert_norm_inequality_first_var}
Assume that $\Phi$ is submultiplicative. If $F \in L_\mu^\Phi(I \times J)$ depends only on the first variable, i.e. $F(x_1,x_2) = g(x_1)$ for some $g \in L_{\mu_1}^\Phi(I)$, then
\[
\Phi^{-1}\left(\mu_2(J)\right)\cdot\left\|g\right\|_{L_{\mu_1}^\Phi(I)}
\leq 
\left\|F\right\|_{L_\mu^\Phi(I \times J)} 
\leq 
\left[\Phi^{-1}\left(\frac{1}{\mu_2(J)}\right)\right]^{-1}\cdot \left\|g\right\|_{L_{\mu_1}^\Phi(I)}.
\]
\end{lemma}
\begin{proof}
The result follows from an argument symmetrical to that of \autoref{lemma:invert_norm_inequality_second_var}
\end{proof}
\newpage
We are now ready to prove our main results.
\begin{proof}[Proof of \autoref{thm:two_dim_sufficiency}]
Let
\[
f_{\nu, \mathsf{av}} 
= \frac{1}{\nu(I \times J)}\int_{I \times J} f(y_1, y_2) \;\dd\nu_1(y_1)\dd\nu_2(y_2).
\]

Using the triangle inequality,
\[
\begin{aligned}
&\left\|f - f_{\nu, \mathsf{av}} \right\|_{L_\mu^\Phi(I \times J)}\\
\leq& 
\left\|f - \frac{1}{\nu_1(I)}\int_I f(y_1,\bullet) \dd\nu_1(y_1) \right\|_{L_\mu^\Phi(I \times J)} 
+ 
\left\|\frac{1}{\nu_1(I)}\int_I f(y_1,\bullet) \dd\nu_1(y_1) - f_{\nu, \mathsf{av}}\right\|_{L_\mu^\Phi(I \times J)}\\
=&: S + T
\end{aligned}
\]
Using \autoref{lemma:invert_norm_inequality} and the $1$-dimensional result \cite[Theorem 1.4]{korobenko_milshstein_yong}, we have
\[
\begin{aligned}
S
=& \left\|f - \frac{1}{\nu_1(I)}\int_I f(y_1,\bullet) \dd\nu_1(y_1) \right\|_{L_\mu^\Phi(I \times J)}\\
\leq& 
\left\|\left\|f - \frac{1}{\nu_1(I)}\int_I f(y_1,\bullet) \dd\nu_1(y_1) \right\|_{L_{\mu_1}^\Phi(I)}\right\|_{L_{\mu_2}^\Phi(J)}\\
&\hspace{4.5in}\text{(\autoref{lemma:invert_norm_inequality})}\\
\leq&
\left\|C_0(\Phi)\cdot K_{p_1, \Phi}(\mu_1, \nu_1, w_1)\cdot\left\|\frac{\partial f}{\partial x_1} \right\|_{L_{w_1}^{p_1}(I)}\right\|_{L_{\mu_2}^\Phi(J)}\\
&\hspace{4.5in}\text{(\cite[Theorem 1.4]{korobenko_milshstein_yong})}\\
=&\;
C_1\cdot \left\|\frac{\partial f}{\partial x_1} \right\|_{L_{w_1 \times \mu_2}^{(p_1, \Phi)}(I \times J)}.
\end{aligned}
\]
\newpage
Also,
\[
\begin{aligned}
T &= \left\|\frac{1}{\nu_1(I)}\int_I f(y_1,\bullet) \dd\nu_1(y_1) - \frac{1}{\nu(I \times J)}\int_{I \times J} f(y_1, y_2) \;\dd\nu_1(y_1)\dd\nu_2(y_2)\right\|_{L_\mu^\Phi(I \times J)}\\
&= \frac{1}{\nu_1(I)} \left\|\int_I f(y_1,\bullet) \dd\nu_1(y_1) - \frac{1}{\nu_2(J)}\int_I \int_J f(y_1, y_2) \;\dd\nu_2(y_2)\; \dd\nu_1(y_1) \right\|_{L_\mu^\Phi(I \times J)}\\
&= \frac{1}{\nu_1(I)} \left\|\int_I \left(f(y_1,\bullet) - \frac{1}{\nu_2(J)} \int_J f(y_1, y_2) \;\dd\nu_2(y_2)\right) \dd\nu_1(y_1) \right\|_{L_\mu^\Phi(I \times J)}\\
&\leq \frac{\left[\Phi^{-1}\left(\frac{1}{\mu_1(I)}\right)\right]^{-1}}{\nu_1(I)} \left\|\int_I \left(f(y_1,\bullet) - \frac{1}{\nu_2(J)} \int_J f(y_1, y_2) \;\dd\nu_2(y_2)\right) \dd\nu_1(y_1) \right\|_{L_{\mu_2}^\Phi(J)}\\
&\hspace{4.5in}\text{(\autoref{lemma:invert_norm_inequality_second_var})}\\
&\leq \frac{2\left[\Phi^{-1}\left(\frac{1}{\mu_1(I)}\right)\right]^{-1}}{\nu_1(I)} \int_I \left\| f(y_1,\bullet) - \frac{1}{\nu_2(J)}\int_J f(y_1, y_2) \;\dd\nu_2(y_2)  \right\|_{L_{\mu_2}^\Phi(J)} \dd\nu_1(y_1) \\
&\hspace{4.5in}\text{(\autoref{lemma:gen_minkowski})}\\
&\leq 
\frac{2\left[\Phi^{-1}\left(\frac{1}{\mu_1(I)}\right)\right]^{-1}}{\nu_1(I)} \cdot C_0(\Phi)\cdot K_{p_2,\Phi}(\mu_2, \nu_2, w_2) \int_I \left\| \frac{\partial f}{\partial x_2}  \right\|_{L_{w_2}^{p_2}(J)} \dd\nu_1(y_1) \\
&\hspace{4.5in}\text{(\cite[Theorem 1.4]{korobenko_milshstein_yong})}\\
&= \frac{2\left[\Phi^{-1}\left(\frac{1}{\mu_1(I)}\right)\right]^{-1}}{\nu_1(I)} \cdot C_2 \cdot \left\| \frac{\partial f}{\partial x_2}  \right\|_{L_{\nu_1 \times w_2}^{(\widehat{1,p_2})}(I \times J)} \\
&\leq \frac{2\left[\Phi^{-1}\left(\frac{1}{\mu_1(I)}\right)\right]^{-1}}{\nu_1(I)^{1/{s_1}}} \cdot C_2 \left\| \frac{\partial f}{\partial x_2}  \right\|_{L_{\nu_1 \times w_2}^{(\widehat{s_1,p_2})}(I \times J)}. \\
&\hspace{4.5in}\text{(H\"older's inequality)}\\
\end{aligned}
\]
\end{proof}
\newpage
\begin{proof}[Proof of \autoref{thm:two_dim_necessity}]
Assume that \autoref{eqn:sufficiencythm} holds for all Lipschitz continuous functions $f\colon I \times J \to \R$. Fix such an $f$, and fix $x_1 \in I$. Define
\begin{align*}
    g \colon J &\longrightarrow \R \\
    x_2 &\longmapsto f(x_1,x_2)
\end{align*}
Note that we may also view $g$ as a (Lipschitz continuous) function $I \times J \to \R$ that depends only on the second input variable. From this point of view, $\frac{\partial g}{\partial x_2} = \frac{\partial f}{\partial x_2}$. Therefore,
\[
\left\|g - \frac{1}{\nu_2(J)}\int_{J} g(y_2)\;\dd\nu_2(y_2)\right\|_{L_\mu^\Phi(I \times J)}
    \leq
    C_2 
    \cdot \frac{2 \left[\Phi^{-1}\left(\frac{1}{\mu_1(I)}\right)\right]^{-1}}{\nu_1(I)^{s_1}}
    \cdot\left\|\frac{\partial f}{\partial x_2}\right\|_{L_{\nu_1 \times w_2}^{(\widehat{s_1, p_2})}(I \times J)}.
\]
Now observe that
\[
\begin{aligned}
&
\left\|g - \frac{1}{\nu_2(J)}\int_{J} g(y_2) \;\dd\nu_2(y_2)\right\|_{L_{\mu_2}^\Phi(J)}\\
\leq& \left[\Phi^{-1}\left(\mu_1(I)\right)\right]^{-1}\cdot \left\|g - \frac{1}{\nu_2(J)}\int_{J} g(y_2) \;\dd\nu_2(y_2)\right\|_{L_\mu^\Phi(I \times J)}\\
&\hspace{4.5in}\text{(\autoref{lemma:invert_norm_inequality_second_var})}\\
\leq& \left[\Phi^{-1}\left(\mu_1(I)\right)\right]^{-1} \cdot \frac{2 \left[\Phi^{-1}\left(\frac{1}{\mu_1(I)}\right)\right]^{-1}}{\nu_1(I)^{s_1}}\cdot C_2\cdot\left\|\frac{\partial f}{\partial x_2}\right\|_{L_{\nu_1 \times w_2}^{(\widehat{s_1, p_2})}(I \times J)}\\
&\hspace{4.5in}\text{(\autoref{eqn:sufficiencythm})}\\
=& \left[\Phi^{-1}\left(\mu_1(I)\right)\right]^{-1} \cdot \frac{2 \left[\Phi^{-1}\left(\frac{1}{\mu_1(I)}\right)\right]^{-1}}{\nu_1(I)^{s_1}}\cdot C_2\cdot\left(\int_I\left\|\frac{\partial f}{\partial x_2}\right\|_{L_{w_2}^{p_2}(J)}^{s_1} \dd\nu_1(x_1)\right)^{1/s_1}\\
=& \left[\Phi^{-1}\left(\mu_1(I)\right)\right]^{-1} \cdot 2 \left[\Phi^{-1}\left(\frac{1}{\mu_1(I)}\right)\right]^{-1}\cdot C_2 \cdot \left\|\frac{\partial f}{\partial x_2}\right\|_{L_{w_2}^{p_2}(J)}.\\
\end{aligned}
\]
Thus we may apply \cite[Theorem 1.6]{korobenko_milshstein_yong} to say that $\tilde{K}_{p_2,\Phi}(\mu_2, \nu_2, w_2) < \infty$.

\smallskip
To show that $\tilde{K}_{p_1,\Phi}(\mu_1, \nu_1, w_1) < \infty$, instead fix $x_2 \in J$, and define $g$ to be
\begin{align*}
    g \colon I &\longrightarrow \R \\
    x_1 &\longmapsto f(x_1,x_2)
\end{align*}
Then, an argument symmetrical to the one above can be made to complete the proof.
\end{proof}
\newpage
\bibliography{biblio}
\bibliographystyle{alpha}
\vspace{+0.25in}
\end{document}